\documentclass[a4paper,11pt]{amsart}

\usepackage{amsmath}
\usepackage{amsthm}
\usepackage{amssymb}
\usepackage{mathptm}
\usepackage{graphicx}

\theoremstyle{plain}
\newtheorem{thm}{Theorem}[section]
\newtheorem{cor}[thm]{Corollary}
\newtheorem{lem}[thm]{Lemma}
\newtheorem{example}[thm]{Example}
\newtheorem{prop}[thm]{Proposition}
\theoremstyle{definition}
\newtheorem{defn}[thm]{Definition}

\theoremstyle{remark}

\title{Units in near-rings}
\author{Tim Boykett}
\address{Institut f\"ur Algebra, Johannes Kepler Universit\"at Linz and Time's Up}
\email{tim.boykett@jku.at and tim@timesup.org}
\urladdr{http://timesup.org}

\author{Gerhard Wendt}
\address{Institut f\"ur Algebra, Johannes Kepler Universit\"at Linz}
\email{Gerhard.Wendt@jku.at}

\urladdr{http://www.algebra.uni-linz.ac.at}
\thanks{This work has been supported by grant P23689-N18 of the Austrian National Science Fund (FWF)}
\begin{document}
\begin{abstract} 
We investigate near-ring properties that generalize nearfield properties about units.
We study  zero symmetric near-rings $N$ with identity with two interrelated properties:
the units with zero form an additive subgroup of $(N,+)$; the units act without fixedpoints on $(N,+)$. 
There are many similarities between these cases, but also many differences. Rings with these properties are fields, near-rings allow more possibilities, which are investigated.
Descriptions of  constructions are obtained and used to create examples showing the two properties are independent but related.
Properties of the additive group as a $p$-group are determined and it is shown that proper examples are neither simple nor $J_2$-semisimple.
\end{abstract}

\maketitle
\section{Introduction}

Within nearring theory, the subjects of planar nearrings and nearfields form a strong core of well-understood structure theory.
Two important tools are used, namely (right) translation maps and the properties of units.

In this paper we investigate generalisations of nearfields and planar nearrings in order to examine the structural implications. 
The first property relates to the additive closure of the units in a nearring, the second property relates to the nature of the right translation maps as fixed point free.
We see similar results proven using similar techniques in each case.

We commence with an introduction to the relevant aspects of nearring theory, namely nearfields, planar nearrings and fixed point free transformations. Then we consider the case of rings and show that our properties lead to rings with either property being fields.
We follow this up with a general construction for nearrings and show a construction technique developed from near-vector-space considerations for examples having both properties, plus two specific examples of near-rings having our two properties independently.

The next section investigates the structure of the additive group for each of these properties. 
We then investigate several special classes of near-rings having the properties, before concluding with some open problems that suggest themselves for future work.

\section{Background}

A \emph{near-ring} $(N,+,*)$ is a group $(N,+)$ with identity $0$ and a semigroup $(N,*)$ with the right distributive law $\forall a,b,c\in N,\, (a+b)*c=a*c+b*c$. 
We use 0-symmetric nearrings such that $0*n=0*n=0$. 
If there is no risk of confusion we will often omit the symbol $*$. We will be mostly concerned with near-rings satisfying the DCCN, the descending chain condition for $N$-subgroups of $N$. 
If a near-ring $N$ is a ring, then the DCCN is the same as the DCCL for rings, the descending chain condition on left ideals. 
However, some results only apply in the finite case.

Let $U$ be the set of all units, i.e.\ invertible elements w.r.t multiplication $*$.
Then $(U,*)$ is a group, the \emph{group of units} of the near-ring. 
A near-field is a near-ring in which the non-zero elements are all units, forming a group under multiplication.

 An automorphism $\phi$ of a group $(N,+)$ is said to have a \emph{fixed point} $n \in N$ if $\phi(n)=n$. If $0$ is the only fixed point of $\phi$, then $\phi$ is called  \emph{fixed point free} (abbreviated by fpf in the following) on $N$. A group of automorphisms $\Phi$ acting on the group $(N,+)$ is called a \emph{fixedpointfree automorphism group} if every non-identity automorphism in $\Phi$ is fpf.

For $u \in N$ let $\psi_u: N \longrightarrow N, n \mapsto n*u$ be the \emph{right translation map} induced by $u$. Since the right distributive law holds, $\psi_u$ is an endomorphism of $(N,+)$. If $u$ is a unit, then  $\psi_u$ is an automorphism. The converse also holds. Let $\psi_r$ be an automorphism. So, there is an element $r_1$ such that $\psi_r(r_1)=r_1*r=1$. Moreover, $r=1*r=r*1=r*(r_1*r)=(r*r_1)*r$. Since $\psi_r$ is injective, $1*r=(r*r_1)*r$ implies $1=r*r_1$ and so $r$ is a unit.

For a set $S \subseteq N$ we define $\Psi_S:=\{\psi_s \mid s \in S\}$. 
Therefore, $(\Psi_U, \circ)$ is a group of automorphisms of $(N,+)$, under function composition.

\begin{lem}\label{Lemma0} Let $N$ be a zero symmetric near-ring with identity $1$ and DCCN. 
Then $n \in N$ has a multiplicative inverse iff $\psi_n$ is injective.
\end{lem}
\begin{proof} Suppose that $n$ is such that $\psi_n$ is injective. By the DCCN there exists a natural number $k$ such that $Nn^k=Nn^{k+1}$. Thus, for all $m \in N$ there exists a $j \in N$ such that $mn^k=jn^{k+1}$. Consequently, $(m-jn)n^k=0$. Since $\psi_n$ is injective, also $\psi_{n^k}$ is injective which results in $m=jn$. Thus, there is an element $l \in N$ such that $1=ln$. Thus  $\psi_l$ is injective and in the same way we again see that there is an element $h \in N$ such that $1=hl$. So we have $h=hln=1n=n$ and we see that $l$ is the inverse of $n$. The rest is clear.
\end{proof}

Let $(N,+,*)$ be a nearring. $(M,+) \leq (N,+)$ is an $N$-group if $\forall n\in N$, $m\in M$, $nm\in M$.
Then $Aut_N(M) = \{f \in Aut(M) \vert \forall n \in N,\, f(nm)=nf(m) \}$.

A \emph{planar nearring} is a nearring $(N,+,*)$ such that every equation of the form $xa=xb+c$ with 
$\psi_a\neq\psi_b$ has a unique solution $x$ and there are at least three distinct right translation maps.
The right translation maps for planar nearrings are all fpf automorphisms of the group $(N,+)$.

Similarly  right translations $\psi_n$ for nonzero elements $n$ of a nearfield are fpf automorphisms.
Note that a planar nearring with identity is a nearfield.

With these notes we have set the ground for our investigations.
In this paper we study near-rings which have special properties concerning their units, which we introduce as follows.

In a near-field, the sum of two units is a unit or zero,
while in general, the units of a near-ring are not closed w.r.t.\ addition. 
On one hand, we are interested in near-rings where $(U \cup \{0\},+)$ is a subgroup of $(N,+)$. 
Since $(U,*)$ is a group, this is equivalent to saying that $(U \cup \{0\},+,*)$ is a subnear-field of $(N,+,*)$. Trivial examples for such type of near-rings are  near-fields but we will see that these are not the only examples.

\begin{defn} An \emph{f-near-ring} $N$ is a zero symmetric near-ring with identity and set of units $U$ where $(\Psi_U, \circ)$ acts as a fixedpointfree automorphism group on the additive group of the near-ring. 
An \emph{a-near-ring} $N$ is a zero symmetric near-ring with identity and set of units $U$ such that $U_0$, the set of units of the near-ring adjoined with the zero of the near-ring, forms a subnear-field of the near-ring w.r.t. the near-ring operations.
Near-rings with both properties will be called \emph{af-near-rings}.
\end{defn}

Note that if the group of units is trivial, i.e.\ consists only of the identity, then the concept of f-near-ring is vacuous. 
In the next section we show that an a-ring with DCCL and with only one unit is ${\mathbb Z}_2^{k}$, $k$ a natural number. a-nearrings with only one unit have elementary abelian 2-groups additively, but otherwise remain open.

In the next section we look at a- and f-near-rings which are rings and show that they are all fields.
Following that, we look at the general structure of a- and f-near-rings and show that many examples exist.

\section{The ring case}

We first study rings which are a-near-rings, f-near-rings respectively. In case we restrict to rings with DCCL we will see that such rings are fields. The authors are not aware of a ring theoretic result which has a similar content as the following Theorem. Therefore we also give a full proof.
 
\begin{thm}\label{Lem1} Let $(R,+,*)$ be a ring with identity and  group of units $U$ such that $|U| \geq 2$. Then the following are equivalent: 
\begin{enumerate} 
\item[a.] $R$ is a field,
\item[b.] $R$ is an a-near-ring with DCCL, 
\item[c.] $R$ is an f-near-ring with DCCL.
\end{enumerate}
\end{thm}
\begin{proof} 
$a \Rightarrow b:$ This is clear from the definition of an a-near-ring and the fact that any field satisfies the DCCL.

$b \Rightarrow c:$ Let $u$ be a non-identity unit in $R$. Suppose that $ru=r$ for some $r \in R$. Thus, $r(1-u)=0$. By assumption, $1-u \in U$ and consequently, $r=0$. This shows that $R$ is an f-near-ring.

$c \Rightarrow a:$  Let $a, b \in U$ and suppose $a-b \not \in U$. Since $R$ has DCCL and $a-b$ is not a unit, $\psi_{a-b}$ cannot be an injective map by Lemma \ref{Lemma0}. Thus, there is an element $z \in R\setminus\{0\}$ such that $0=z(a-b)=za-zb$. Consequently, $za=zb$ and therefore, $z=zba^{-1}$. So, $z$ is a fixed point. Since $ba^{-1} \in U$, by assumption, $ba^{-1}=1$ and $a=b$. Therefore, given $a, b\in U$ such that $a \neq b$, $a-b \in U$. 

Suppose there is a nilpotent element $a \in R$, so $a^{n}=0$ for some non-zero integer $n$. Then, $(1-a)(1+a+\ldots +a^{n-1})=1$, so $1-a$ is a unit, $u_1$ say. Consequently, $1-u_1=a \not \in U$. Since $1 \neq u_1$ implies $1-u_1 \in U$, as just shown, $1=u_1$ and $a=0$. This shows that $R$ does not contain non-zero nilpotent elements, so it is a reduced ring.

Thus, by \cite[Theorem 9.41]{Pilz1} $R$ is a finite direct product of $k$ fields, say. Suppose that $k \geq 2$. Then there exist orthogonal idempotents being the identities of the fields in the direct product such that $1=e_1+ \ldots + e_k$. By assumption, there exists a non-identity unit $u=f_1 + \ldots + f_k \in R$. W.l.o.g.\ we assume that $f_1 \neq e_1$. Then $w=f_1+ \ldots + e_k$ is also a unit which does not equal the identity. Then, $e_2w=e_2(f_1+ \ldots + e_k)=e_2$ and therefore, $e_2$ is a fixed point of $\psi_{w}$ which contradicts the assumption. Thus, $k=1$ and $R$ is a field.
\end{proof}

What happens in rings with identity and with DCCL which have the identity element as their single unit element answers the next proposition.

\begin{prop}\label{Ring1} Let $R$ be a ring with identity, which is the only unit element in $R$, $R$ satisfying the DCCL. Then, $R$ is isomorphic to a direct product of copies of the field $\Bbb{Z}_2$.
\end{prop}
\begin{proof} Suppose $R$ contains a nilpotent element $a$, say. Then, $(1-a)(1+a+ \ldots + a^{n-1})=1$ for some natural number $n>1$. Hence, $1-a$ is a unit and by assumption, $1-a=1$. Thus, $a=0$ which shows that $R$ is a reduced ring. By \cite[Theorem 9.41]{Pilz1} a reduced ring with DCCL is isomorphic to a direct product of fields. Since we have only $1$ as unit element in $R$, any field in the direct composition has only one unit element which shows that $R$ is isomorphic to a direct product of copies of ${\mathbb Z}_2$.
\end{proof}

The situation in near-rings with identity having precisely the identity as unit element is more complex and a complete solution of how these type of near-rings look like is not known to the authors. We will look at some concrete examples using  Theorem \ref{Tim1}.

Theorem \ref{Lem1} tells us that within the class of rings with DCCL there do not exist a-rings or f-rings other than fields. 
We have to look at proper near-rings in order to get nontrivial examples, which we do in the following section.

\section{Constructions  using semigroups of group endomorphisms}\label{NVSsec}
In this section we present a  method of constructing all a- and f-near-rings.

\begin{thm}\label{A1} The following are equivalent:
\begin{enumerate}
\item $(N,+,\cdot)$ is an  f-near-ring.
\item There exists \begin{enumerate} \item a group $(M,+)$. \item a semigroup $S= G \cup E$ of group endomorphisms of $(M,+)$ where $G$ is a group of fpf automorphisms of $(M,+)$ and $E$ is a semigroup of non-bijective endomorphisms containing the zero map $\overline{0}$.
\item an element $m \in M$ such that $S(m)=M$ and for all $n \in M$ there is a unique $s_n \in S$ with $n=s_n(m)$.

\end{enumerate} such that $(N,+,\cdot) \cong (M,+,*)$, where $a*b=s_b(a)$ ($a, b \in M$).
\end{enumerate}
\end{thm}
\begin{proof} $(1) \Rightarrow (2)$: Let $(M,+)=(N,+)$. 
For all $r \in N$, $\psi_r :N \longrightarrow N, n \mapsto nr$ is an endomorphism of $(N,+)$. Let $S:=\{\psi_r \mid r \in N\}$, then $S$ is a semigroup of group endomorphisms. $G:=\Psi_U=\{\psi_u \mid u \in U\}$ is a subgroup of $S$ acting without fixedpoints on $(N,+)$. Let $E:=\{\psi_z \mid z \in N \setminus U\}$. Then $S = G \cup E$ and no element in $E$ is a bijection (certainly, $\overline {0} \in E$). Let $m:=1$. For each $n \in N$ we have $n=1*n=\psi_n(1)$ with $\psi_n \in S$. Certainly, $\psi_n$ is the unique element $s$ in $S$ such that $s(1)=n$. $S(1)=N$ and for each $n \in N$ there is a unique $s_n \in S$ such that $n = s_n(1)$, namely $s_n=\psi_n$. Note that for $a, b \in N$ we have $a \cdot b =\psi_b(a)=s_b(a)=a*b$.

$(2) \Rightarrow (1)$: We first show that $*$ is indeed a near-ring multiplication. By assumption, we know that for a given element $b \in M$ there is a unique element $s_b \in S$ such that $b=s_b(m)$, so multiplication is well defined. We first prove associativity of $*$: On one hand, $(a*b)*c=s_b(a)*c=s_c(s_b(a))$. Since $b=s_b(m)$ and $c=s_c(m)$ we have $a*(b*c)=a*s_c(b)=a*s_c(s_b(m))=s_c(s_b(a))$.

Using that $s_c$ is a group endomorphism for given $c \in M$ we get $(a+b)*c=s_c(a+b)=s_c(a)+s_c(b)=a*c+b*c$. This shows that $*$ is a right distributive multiplication and $(M,+,*)$ is a near-ring.

As a group $G$ contains the identity function $id$, we have  $m=id(m)$ and so we have $a*m=a$. On the other hand, $m*a=s_a(m)=a$ which proves that $m$ is the identity element w.r.t $*$. Since $\overline{0} \in E$ ($\overline{0}$ being the zero map) we have $0 = \overline{0}(m)$ and consequently, $n*0=\overline{0}(n)=n$ for all $n \in M$ and hence, $M$ is a zero symmetric near-ring.

Let $U$ be the group of units of $M$. We now prove $U=G(m)$:

$G(m) \subseteq U$, because for all elements $g(m) \in G(m)$ we have that $g(m)*g^{-1}(m)=g^{-1}g(m)=m$, $m$ being the identity element, so $g^{-1}(m)$ is the inverse w.r.t $*$. 

Let $u \in U$, so there is an element $u^{-1} \in M$ such that $u^{-1}*u=m$. 
We know that there is a unique element $s_u \in S$ such that $s_u(m)=u$ and also $s_{u^{-1}}(m)=u^{-1}$. Therefore, $u^{-1}*u=s_u(s_{u^{-1}}(m))=m$. Since $S$ is a semigroup, we have $s_u \circ s_{u^{-1}} \in S$. So, we have $m=id(m)=s_us_{u^{-1}}(m)$ and this implies $s_u\circ s_{u^{-1}}=id$ because there is only one endomorphism in $S$ mapping $m$ to $m$, by assumption. Suppose $s_u \in E$. Since $id \in G$, $s_{u^{-1}} \not \in E$ and consequently, $s_{u^{-1}} \in G$. But then $s_u=s_{u^{-1}}^{-1}\in G$, a contradiction. This shows that $s_u \in G$. This finally gives us $U \subseteq G(m)$.

To finish the proof we have to show that $\Psi_U$ acts without fixedpoints on $(M,+)$. Let $m \neq b=s_b(m)$ where $s_b \in G\setminus \{id\}$. Suppose there is an element $a \in M$ such that $a*b=a$. Then we have that $s_b(a)=a$. Since each non-identity automorphism in $G$ acts without fixedpoints on $M$, we must conclude $a=0$.
\end{proof}

Theorem \ref{A1} will now be used to give an example of an f-near-ring which is not an a-near-ring. More examples for constructing f-near-rings will follow.

\begin{example}\label{UnitEx1}
Let $(M,+):=(\Bbb{Z}_3^{4},+)$ be the $4$-dimensional direct sum of the group $({\mathbb Z}_3,+)$ written as column vectors. 
In the following, the matrices and vectors will be over the field ${\mathbb Z}_3$. 
Let $m:=(0,0,0,1)^t$. 

Let $\Gamma = \{(0,0,0,1)^t, (0,0,0,2)^t, (0,0,1,0)^t, (0,0,2,0)^t\} \subseteq M$, then define

\begin{align*}
%  G&:=\{\left(\begin{array}{cccc}
% 1&0& 0 & 0\\ 0 & 1 & 0 & 0\\ 0 & 0 & 1 & 0 \\ 0 & 0 & 0 & 1
% \end{array}\right), \left(\begin{array}{cccc}
% 2 &0 & 0 & 0\\ 0 & 2 & 0 & 0\\ 0 & 0 & 2 & 0 \\ 0 & 0 & 0 & 2
% \end{array}\right), \left(\begin{array}{cccc}
% 0&-1& 0 & 0\\ 1 & 0 & 0 & 0\\ 0 & 0 & 0 & -1 \\ 0 & 0 & 1 & 0
% \end{array}\right), \left(\begin{array}{cccc}
% 0&1& 0 & 0\\ -1 & 0 & 0 & 0\\ 0 & 0 & 0 & 1 \\ 0 & 0 & -1 & 0
% \end{array}\right) \}
 G&:=\left<\left(\begin{array}{cccc}
0&-1& 0 & 0\\ 1 & 0 & 0 & 0\\ 0 & 0 & 0 & -1 \\ 0 & 0 & 1 & 0
\end{array}\right)\right> \leq GL(3,4)
\\
E&:= \{\left(\begin{array}{cccc}
0&0& n_2 & n_1\\ 0 & 0 & -n_1 & n_2\\ 0 & 0 & n_4 & n_3 \\ 0 & 0 & -n_3 & n_4
\end{array}\right) |  (n_1, n_2, n_3, n_4)^t  \in M \setminus \Gamma\}
\end{align*}

One checks that $G$ acts without fixed points  on $M$ and is a subgroup of order $4$ of the general linear group $GL(3,4)$. By straightforward calculation we see that $E$ is a matrix semigroup. 
Also, one checks that for given matrix $A \in G$ and $B \in E$, the product $AB$ and $BA$ of the matrices is contained in $E$. Let $S:= G \cup E$. 
It follows that $S$ is a matrix semigroup. 
A matrix $s \in S$ induces an endomorphism on $(M,+)$ simply by matrix multiplication. 
Thus (without distinguishing between the matrix and the induced endomorphism) $S$ is a semigroup of group endomorphisms of $M$. 
Also, we see that $G \cdot m = \Gamma$ 
and $E\cdot m = M\setminus \Gamma$. 
Thus, for each element  $n \in M$, there is a unique matrix $s \in S$ such that $n=s \cdot m$.

We now define the multiplication $*$ on $M$ as $a*b=a*(s\cdot m):=s \cdot a$, $(a, b \in M)$.

By Theorem \ref{A1}, $(M,+,*)$ is an f-near-ring. The units $U$ of this near-ring are $U:=\Gamma$. We clearly have that $U_0$ is not additively closed.
\end{example}

%Later in Theorem \ref{PK1} we will see that Example \ref{UnitEx1} is minimal for obtaining an f-near-ring which is not an a-near-ring. 

A similar structure exists for a-near-rings.

\begin{thm}\label{A2} The following are equivalent:
\begin{enumerate}
\item $(N,+,\cdot)$ is an  a-near-ring.
\item There exists \begin{enumerate} \item a group $(M,+)$ with a subgroup $H \subseteq M$. 
\item a semigroup $S=G \cup E$ of group endomorphisms of $(M,+)$ where $G$ is a group of automorphisms and $E$ is a semigroup of non-bijective endomorphisms containing the zero map $\overline{0}$.
\item an element $m \in M$ such that $S(m)=M$, $G(m)=H \setminus \{0\}$ and for all $n \in M$ there is a unique $s_n \in S$ with $n=s_n(m)$.
\end{enumerate} such that $(N,+,\cdot) \cong (M,+,*)$, where $a*b=s_b(a)$ ($a, b \in M$).
\end{enumerate}
\end{thm}
\begin{proof}
$(1) \Rightarrow (2)$: As in the proof of Theorem \ref{A1} let $(M,+)=(N,+)$ and $S:=\{\psi_r \mid r \in N\}$. $G:=\Psi_U=\{\psi_u \mid u \in U\}$ is a subgroup of $S$ acting on $(N,+)$. Let $E:=\{\psi_z \mid z \in N \setminus U\}$. Then $S = G \cup E$ is a semigroup and no element in $E$ is a bijection (certainly, $\overline {0} \in E$). We now let $H=U_0$, so $G(1)=U$. As in the proof of $(1) \Rightarrow (2)$ of Theorem \ref{A1} we let $m=1$ and see that property (c) holds.

$(2) \Rightarrow (1)$: To prove that $(M,+,*)$ is a near-ring runs as in the proof of Theorem \ref{A1}. Also as in the proof of Theorem \ref{A1} we see that $G(m)=H\setminus \{0\}$ is  the set of units of the near-ring. 

It only remains to show that the units are additively closed. Let $a=g_1(m)$ and $b=g_2(m)$ be two units, $a \neq b$. Then we know that $a-b=g_1(m)-g_2(m) \in G(m)$ since $H$ is a subgroup of $M$. Thus, $a-b$ is a unit. 
\end{proof}

We construct an example of a proper a-near-ring.

\begin{example}\label{TG1} We let $M$ be the group which is given by the following multiplicative presentation (the finite Heisenberg group on $p=3$). \[M:=\left<x,y,z: x^{3}=y^{3}=z^{3}=1, yz=zyx, xy=yx, xz=zx\right>\]
$M$ is a non-abelian group of order $27$. Let $S:=\{s_a | a \in M\}$ where $s_a: M \rightarrow M$ is a group endomorphism defined on the generators $x,y,z$ of the group as follows. For $a \in M\setminus \{z, z^{2}\}$ we define $s_a(x):=1, s_a(y):=1, s_a(z):=a$, otherwise $s_z(x):=x, s_z(y):=y, s_z(z):=z$ and  $s_{z^{2}}(x):=x^{2}, s_{z^{2}}(y):=y, s_{z^{2}}(z):=z^{2}$. The endomorphisms $s_z$ and $s_{z^{2}}$ are automorphisms of the group, $s_z=id$ the identity. 
We now let $G=\{s_z, s_{z^{2}}\}$ and $E=\{s_a | a \in M\setminus \{z, z^{2}\}\}$. Hence, $G(z)=\{z,z^{2}\}$ and consequently, $G(z) \cup \{1\}$ is a subgroup $H$ of $M$. By taking $m=z$ we can now apply Theorem \ref{A2} and see that $(M,+,*)$ is an a-near-ring.
\end{example}
Example \ref{TG1} is $LibraryNearRingWithOne(GTW27_{-}4,1)$ in the GAP package SONATA \cite{Sonata}.

 We now use Theorems \ref{A1} and \ref{A2} to construct af-near-rings which are not near-fields such that the units are isomorphic to a given near-field.

\begin{thm}\label{Tim1}
 Let $(F,+,*)$ be a near-field and $(V,+)=(F^k,+)$ the $k$ dimensional direct sum of $(F,+)$. For $i \in \{1, \ldots ,k\}$ let $\alpha_i: F \rightarrow F$ be zero preserving maps which are multiplicative automorphisms of the group $(F\setminus \{0\},*)$ where $\alpha_k$ is the identity map.
 For $\underline a, \underline x \in V$ define the operation $*_1$ in $V$ in the following way:
 \begin{align*}
  \underline x *_1 \underline a &:= (x_1* \alpha_1(a_{k}),\ldots, x_{k-1}*\alpha_{k-1}(a_k),x_k*a_k)\mbox{ if } a_1=a_2 = \ldots=a_{k-1}=0 \\
  \underline x *_1 \underline a &:= (x_k*a_1,\ldots,x_k*a_k)\mbox{ otherwise}
 \end{align*}

Then $(V,+,*_1)$ is an  af-near-ring with units $U=\{(0,\ldots,0,a)\vert a \in F,\,a\neq 0\}$, identity $(0,\ldots,0,1)$, 
 and $(U\cup \{0\},+,*_1)$ isomorphic to $(F,+,*)$.
\end{thm}
\begin{proof}
Let $s:V\rightarrow End(V)$ be defined by $s(\underline a) (\underline x) := \underline x *_1 \underline a$.
By right distributivity of the near-field multiplication, $s(\underline a)$ is a homomorphism.

We first show that $S = \{s(v)\vert v\in V\}$ is closed under composition of functions and is a semigroup.  
For the rest of this proof, we write $a*b$ as $ab$ to aid readability. 
We now show $S$ is closed with 4 cases depending upon the structure of $a$ and $b$.
Let $Z = \{\underline a \in V \vert a_1=\ldots=a_{k-1}=0\}$.

Case 1: Let $\underline a, \underline b \in V\setminus Z$. 
Then
\begin{align*}
 (s(\underline a)\circ s(\underline b))  (\underline x) &= s(\underline a)(s(\underline b(\underline x)) \\
   &= s(\underline a)(x_kb_1,\ldots, x_kb_k)\\
   &=(x_kb_ka_1,\ldots, x_kb_ka_k)
\end{align*} 
If $b_k=0$ this is the zero map $s((0,\ldots,0))$, otherwise there is some nonzero  $b_ka_1,\ldots, b_ka_{k-1}$. 
Thus, $(x_kb_ka_1,\ldots, x_kb_ka_k)=s((b_ka_1,\ldots,b_ka_{k}))(x_1, \ldots, x_{k})$, so $S$ is closed.

Case 2: Now consider $\underline a, \underline b \in Z$. 
We get  
\begin{align*}
 (s(\underline a) \circ s(\underline b)) (\underline x ) = s(\underline a)(x_1\alpha_1(b_k), \ldots, x_kb_k)=(x_1\alpha_1(b_k)\alpha_1(a_k), \ldots, x_kb_ka_k)
\end{align*} 
From the fact that $\alpha_i$,$i \in \{1, \ldots, k\}$ are multiplicative homomorphisms we have that  
\begin{align*}
 (x_1\alpha_1(b_k)\alpha_1(a_k), \ldots, x_kb_ka_k)
    &=(x_1\alpha_1(b_k a_k), \ldots, (x_{k-1}\alpha_{k-1}(b_k a_k),x_kb_ka_k)
 \\ &=s((0, \ldots,0, b_ka_k))(\underline x)
\end{align*}
 so $S$ is closed.

Case 3: Let $\underline a \in V\setminus Z$, $\underline b \in Z$.  
We calculate:  
\begin{align*}
 (s(\underline a)\circ s(\underline b))  (\underline x) &=s(\underline a)(x_1\alpha_1(b_k), \ldots, x_kb_k) \\ &=(x_kb_ka_1, \ldots, x_kb_ka_k)
\end{align*}
If $b_k=0$, this is the zero map $s((0,\ldots,0))$. In case $b_k \neq 0$ there is some nonzero  $b_ka_1,\ldots, b_ka_{k-1}$,
so we get $(x_kb_ka_1, \ldots, x_kb_ka_k)=s((b_ka_1,\ldots,b_ka_{k}))(\underline x)$, so $S$ is closed.

Case 4: Let $\underline a \in Z$, $\underline b \in V\setminus Z$. 
This gives:
\begin{align*}
 (s(0,\ldots,0, a_{k})\circ s(\underline b))  (\underline x) &=
   s(0,\ldots,0,a_{k})(x_kb_1, \ldots, x_kb_k) \\ &=(x_kb_1\alpha_1(a_k), \ldots, x_kb_ka_k)
\end{align*}
If $a_k=0$ this is the zero map, since the $\alpha_i$,$i \in \{1, \ldots, k\}$ are zero preserving. 
If $a_k \neq 0$ we use that the maps  $\alpha_i$,$i \in \{1, \ldots, k\}$ are automorphisms to see that there is some nonzero  $b_1\alpha_1(a_k),\ldots, b_{k-1}\alpha_{k-1}(a_k)$. 
So, we get $(x_kb_1\alpha_1(a_k), \ldots, x_kb_ka_k)=s((b_1\alpha_1(a_k), \ldots, b_ka_k))(\underline x)$, so $S$ is also closed. 

Thus we see  that $(S, \circ)$ is a semigroup.

For all $a\in F\setminus \{0\}$, $s((0,\ldots,0,a))$ is a bijection, thus an automorphism. 
This gives us our group $G := \{s((0,\ldots,0,a))\vert a\in F^*\}$.
Since $F$ is a near-field, $s((0,\ldots,0,a))$ is fpf or $a=1$.
For $\underline a \in V\setminus Z$, $s(\underline a)$ is not surjective, thus a proper endomorphism.

Also $s((0,\ldots,0))(\underline x) = (0,\ldots,0)$ is the zero map.

For all $\underline x \in V$, $s(\underline x)(0,\ldots,0,1) = \underline x$.
Thus $S = \{s(\underline v)\vert \underline v\in V\}$ satisfies the requirements of Theorem \ref{A1}, so we have an f-near-ring.
Noting that $G(m) = 0\times\ldots\times 0 \times F \leq V$ we have the requirements for Theorem \ref{A2} so we have an a-near-ring and thus an af-near-ring.
\end{proof}

The multiplication $*_1$ of Theorem \ref{Tim1} is similar to the action of near-vector spaces, see \cite{Howell} for details. 
We do not follow this line of discussion and possible links to the near-vector space construction here.

Theorem \ref{Tim1} can be efficiently used to construct examples of af-near-rings. 
Take a field $F$, $k = 2$ and $\alpha_1 =id$. 
Then $(V,+,*_1)$ as in Theorem \ref{Tim1} is an af-near-ring which is not a near-field and also not a ring. 
We might take $\alpha_1$ to be the Frobenius automorphism of the field F and we get another example of af-near-rings not being near-fields. 
We might also take $k=3$, $\alpha_1=id$ and $\alpha_2$ to be the Frobenius automorphism to get examples of af-near-rings of higher order. Also $\alpha_1=\alpha_2$ can be taken. In this manner we see that we can construct af-near-rings of all possible (finite) orders (that a finite f-near-ring must have an order of a  power of a prime follows from the next section). 

We can similarly take $F$ to be a proper nearfield and use near-field automorphisms to obtain yet more examples.

To characterize when two af-near-rings are isomorphic is an open question.

More examples using Theorem \ref{A1} and \ref{A2} will show up in the next section.

Theorem \ref{Tim1} also allows us to explicitly construct  near-rings with identity which have only the identity as single unit element. 
Take $k$ a natural number and $F={\mathbb Z}_2$ and $\alpha_i =id$ for all $i \in \{1, \ldots , k-1\}$. Then $(V,+,*_1)$ is a near-ring with one single unit and it is not a ring. 

If each zero symmetric near-ring with identity which has only one unit is of the form as constructed by Theorem \ref{Tim1} is not known to the authors. At least we can say that their additive groups  are elementary abelian 2-groups and we can say more about their structure when they are $2$-semisimple in  Proposition \ref{NRU1}.

\begin{prop} Let $N$ be a zero symmetric near-ring with identity, which is the only unit element in $N$. 
Then $(N,+)$ is an elementary abelian 2-group.
\end{prop}
\begin{proof} Since $(-1)(-1)=1$, $-1$ is also the only unit, so $-1=1$. Thus, for any $n \in N$, $n+n=(1+1)n=0$ which shows that $(N,+)$ is an abelian group of exponent $2$, thus an elementary abelian 2-group.
\end{proof}

In the next section we look further at the additive group of a- and f-near-rings.

\section{The additive group}

Let $(N,+)$ be a group. As usual, the exponent of the group is the least common multiple of all the orders of the group elements. For $k \in N$, let $ord(k)$ be the order of $k$ w.r.t.\ the group operation $+$. Then we have the following.
\begin{prop} \cite[Proposition 9.111]{Pilz1}\label{TU1} Let $N$ be a finite near-ring with identity $1$. Let $n$ be the exponent of $(N,+)$. Then $ord(1)=n$.
\end{prop}

\begin{thm}\label{thm1} 
Let $(N,+,*)$ be a finite a-near-ring. Then, $(N,+)$ is a group of exponent $p$, $ord(1)=p$, $p$ a prime number.
\end{thm}
\begin{proof} Let $n$ be the exponent of $(N,+)$. We have to show that $n=p$, $p$ a prime number.
Consider the cyclic group $(\left<1\right>,+)$ (additively) generated by $1$.  By assumption, $(\left<1\right> \cup \{0\}) \subseteq U \cup \{0\}$. Let $\cdot $ be the usual product in the set of integers and $m,k$ be two positive integers. Let $m\cdot 1=1+\ldots + 1$ ($m$ times $1$) and $k\cdot 1=1+\ldots +1$ ($k$ times $1$) be two elements of $\left<1\right>$. By the right distributive law in $N$ and by using the fact that $1$ is the identity we get $(m\cdot 1)*(k\cdot 1)=(m\cdot k)\cdot 1$. This  is certainly contained in  $\left<1\right>$. 
Consequently, $S:=(\left<1\right>, +, *)$ is a subnear-ring of $N$ containing the identity (see also \cite{Maxson1}), even more, it is a subnear-ring of the  subnear-field $(U \cup \{0\}, +, *)$ of $N$. Actually, $S$ is a ring. It is easy to see that $S$ is  a ring isomorphic to $\Bbb{Z}_n$ since the function $\psi: S \longrightarrow \Bbb{Z}_n, m\cdot 1 \mapsto [m]_n$ is a near-ring isomorphism ($[m]_n$ is the usual congruence modulo $n$). For completeness we give a proof.

$\psi$ is well defined: Let $m\cdot 1 = k\cdot 1$ and suppose $k>m$. Then $(k-m)\cdot 1=0$. Since $n$ is the additive order of $1$ by Proposition \ref{TU1}, $n \mid (k-m)$ and therefore $[k-m]_n=0$. So $[k]_n=[m]_n$ and $\psi$ is well defined. Furthermore, $\psi(m\cdot 1 + k\cdot 1)=\psi((m+k)\cdot 1)=[m+k]_n=[m]_n+[k]_n=\psi(m\cdot 1)+\psi(k\cdot 1)$. Also, $\psi((m\cdot 1)*(k\cdot 1))=\psi(( m\cdot k)\cdot 1)=[m\cdot k]_n=[m]_n\cdot [k]_n$. Consequently, $\psi$ is a near-ring homomorphism. Let $m\cdot 1$ be in the kernel of $\psi$. Then, $\psi(m\cdot 1)=[m]_n=[0]_n$ and consequently $n \mid m$ and we have $m\cdot 1=0$ since $n$ is the exponent of $(N,+)$. $\psi$ certainly is surjective, so we see that $S$ is isomorphic to $\Bbb{Z}_n$. 

Since $S$ is contained in a finite near-field, every non-zero element $s \in S$ has a finite multiplicative order $m$. 
But then $s^{m-1}s=1$ and $s^{m-1} \in S$, so each element in $S$ is invertible w.r.t. multiplication. So, $S$ is a field. Therefore, $n=p$ for some prime $p$. 
\end{proof}

The situation is even clearer for f-near-rings.

\begin{thm}\label{CorA} Let $(N,+,*)$ be a finite f-near-ring. Then, $(N,+)$ is an elementary abelian group.  
\end{thm}
\begin{proof} Let $n$ be the exponent of $(N,+)$. If $n=2$, then for all $m \in N$ we have $0=(1+1)m=m+m$, so $(N,+)$ is elementary abelian. 

Now we assume that $n \neq 2$. 
As in the proof of Theorem \ref{thm1} we consider the subring $S$, generated by the identity element $1$. Let $U_S$ be the group of units of $S$. Since $n\neq 2$ and $S$ is isomorphic to the ring $\Bbb{Z}_n$ we have $\mid U_s\mid \geq 2$. Since $S$ and $N$ have the same identity, an element $s$ being invertible in $S$ is invertible in $N$. Therefore $U_S \subseteq U$. By assumption, $\Psi_U$ does not have non-zero fixed points in the set $N$ and consequently, $\Psi_{U_s}$ does not have non-zero fixed points in $S$. Therefore, $(\Psi_{U_s},\circ)$ is a group of fpf automorphisms of $(S,+)$. Consequently, $S$ fulfilles the assumptions of Theorem \ref{Lem1} and it follows that $S$ is a field. Therefore, $n=p$ for some prime $p$. 

Since $-1 \in S$ and $1 \neq -1$, we have that $-1 \in U_S \subseteq U$. Therefore, $\psi_{(-1)} \in \Psi_U$ and by assumption $\psi_{(-1)}$ is a fpf automorphism and the order of $\psi_{(-1)}$ is two. Groups admitting fpf automorphisms of order $2$ are known to be abelian (see for example \cite[Chapter 10, Theorem 1.4]{Gor}). The proof is finished.
\end{proof}

From this result  we know that Example \ref{TG1} is a proper a-near-ring, because the additive group is not abelian.

While not every f-near-ring is additively closed, there is an additive closure property.

\begin{cor}
The set of units of a finite f-near-ring is a union of cyclic subgroups  (all of the same prime order $p$) of the additive group.
\end{cor}
\begin{proof}
Let $N$ be a f-near-ring, $u \in N$ be a unit.
Note that by the proof of Theorem \ref{CorA}, the  subring $S$ generated by $1\in N$ is a prime field and thus all nonzero elements in $S$ are units.
Then $u+...+u=(1+...+1)u$ so $u+...+u$ is a unit, so $U$ is closed under the taking of cyclic subgroups.
\end{proof}

\section{Special types of a- and f-near-rings}
In this section we investigate several special classes of a- and f-near-rings.

\subsection{Simple a-near-rings and f-near-rings}
We study simple and semisimple a- and f-near-rings with DCCN and see that apart from one exceptional case in the class of {a-near-rings}, these are near-fields.

\begin{thm}\label{Theorem1} Let $(N,+,*)$ be a simple a-near-ring with DCCN and $U$ its set of units. If $N$ is a ring with $|U|=1$, then $N$ is isomorphic to a direct product of the fields $\Bbb{Z}_2$. If $N$ is a ring with $|U| \geq 2$, then $N$ is a field. If $N$ is not a ring, then $N$ is a near-field or $N \cong M_0(\Bbb{Z}_3)$.
\end{thm}

\begin{proof}  From Theorem \ref{Lem1} and Proposition \ref{Ring1} the results concerning the ring case follow.  So we suppose that $N$ is not a ring in the following. 

Due to the DCCN there exists a minimal $N$-subgroup $M$ of $N$. Since $N$ is assumed to be simple, $N$ acts faithfully on $M$, so $(0:M):=\{n\in N \mid \forall m \in M: nm=0\}=\{0\}$. 
Due to the fact that $N$ has an identity element, we have that $Nm \neq \{0\}$ for each non-zero element $m \in M$ and so, $Nm=M$. Consequently, $N$ acts $2$-primitively on $M$ (see \cite[Corollary 4.47]{Pilz1}). 
For $m_1, m_2 \in M$ we define $m_1 \sim m_2$ iff $(0:m_1)=(0:m_2)$. $\sim$ is an equivalence relation in $M$ (see \cite[Remark 4.20]{Pilz1}). Since $N$ has DCCN there is only a finite number of represantatives of $\sim$ by \cite[Theorem 4.46]{Pilz1}. Let $m_1, m_2, \ldots , m_n$ ($n$ a natural number) be a complete set of representatives of the non-zero equivalence classes of $\sim$ in $M$.  Since $N$ has an identity element only the zero $0$ is equivalent to $0$. 

We now assume that $n > 2$, so there exist two different non-zero elements $m_1, m_2$ in $M$ such that $m_1 \not \sim m_2$. By \cite[Theorem 4.30]{Pilz1} there is an element $k \in N$ such that $km_1=m_2$, $km_2=m_1$ and $km_i=m_i$ for $i \geq 3$. $k \neq 1$ because $m_1 \neq m_2$.

Let $a \in (0:k):=\{n \in N \mid nk=0\}$. Thus, for each $i \in \{1, \ldots n\}$, $am_i=0$ (because $am_1=akm_2=0$ and so on). 
Let $m \in M\setminus\{0\}$. Therefore, there is an element $j \in \{1, \ldots , n\}$ such that $m \sim m_j$. 
Since $a \in (0:m_j)$, we now also have $a \in (0:m)$ and so we must have $am = 0$ for each $m \in M$. 
Since $M$ is faithful, we have $a=0$. Thus, $(0:k)=\{0\}$ and consequently, the map $\psi_k$ is injective. Lemma \ref{Lemma0} shows that $k$ is a unit.

From $km_i=m_i$ for $i \geq 3$ we get $(k-1)m_i=0$. By assumption $k-1$ is a unit or zero. The latter case cannot be, so $k-1$ is a unit which implies $m_i=0$. This contradicts the fact that $m_1, m_2, \ldots , m_n$, $n>2$, is a complete set of representatives of the non-zero equivalence classes of $\sim$ in $M$. Hence, $n \leq 2$. 

Suppose $n=1$, so there is only one non-zero equivalence class w.r.t $\sim$ in $M$.
Let $m \in M\setminus \{0\}$ and let $a \in (0:m)$. Then, for any other non-zero element $n \in M$ we must have $a \in (0:n)=(0:m)$. Thus, $aM=\{0\}$ and by faithfulness of $M$ we get that $a=0$. 
Thus, $\psi_m$ is injective for any non-zero element $m \in M$. So, any $m \in M\setminus \{0\}$ is a unit in $N$ by Lemma \ref{Lemma0}. Since $NM \subseteq M$ we must have $1 \in M$ and therefore $N=M$. Thus, $N$ is a near-field.

Finally, suppose $n=2$: $N$ is a $2$-primitive near-ring with identity acting $2$-primitively on $M$. 
Hence, by \cite[Theorem 4.52]{Pilz1} $N \cong M_{G^{0}}(M)$ where $G^{0}:=Aut_N(M)\cup \{\overline 0\}$, $\overline 0$ being the zero function on $M$ and $G:=Aut_N(M)$ ($G$ may be just the set containing the identity map). $G$, the set of $N$-automorphisms, acts without fixed points on $M$. Since $\forall m \in M\setminus\{0\}: Nm=M$, by \cite[Proposition 4.21]{Pilz1} we have that for all $a, b \in M\setminus\{0\}$: $a\sim b$ $\Leftrightarrow$ $G(a)=G(b)$. Thus, $G$ has two non-zero orbits on $M$ with orbit representatives $e_1$, $e_2$, say.  Let $f: \{e_1, e_2\} \longrightarrow M$ be a function. $G$ acts without fixedpoints on $M$ so by \cite[Theorem 4.28, Proposition 7.8]{Pilz1} this function can be uniquely extended to a function $\overline{f} \in M_{G^{0}}(M)$ by defining $\overline{f}(0):=0$, $\overline{f}(m):=g(f(e_1))$ in case $m = g(e_1) \in G(e_1)$ and in case $m = g(e_2) \in G(e_2)$, $\overline{f}(m):=g(f(e_2))$.

We now assume that $|G| \geq 2$ and let $id \neq g \in G$. Let $h: \{e_1, e_2\} \longrightarrow M$ be the function $h(e_1)=e_1$ and $h(e_2)=g(e_2)$ and $\overline{h}$ be its extension in $M_{G^{0}}(M)$. Let $i: \{e_1, e_2\} \longrightarrow M$ be the function $i(e_1)=e_1$ and $i(e_2)=g^{-1}(e_2)$ and $\overline{i}$ its extension in $M_{G^{0}}(M)$. Then, $i$ is the inverse to $f$ w.r.t. the near-ring multiplication (which is function composition) in $M_{G^{0}}(M)$. On the other hand, $(id-\overline{h}) \in M_{G^{0}}(M)$ and $(id-\overline{h})(e_1)=0$. Thus, the non-zero function $(id-\overline{h})$ is not injective and consequently not a unit in $M_{G^{0}}(M)$. But this contradicts the assumption that $N$ is an a-near-ring.

From this contradiction it follows that $|G|=1$, so $G$ only contains the identity map. But we have that $G$ has two non-zero orbits in $M$ and this implies that $M$ is a group with $3$ elements. Hence, $N \cong M_0(\Bbb{Z}_3)$. $N$ has $9=3^{2}$ elements and two units, namely the identity function and the function switching $1$ and $2$. It is easy to see that in this case the units with zero form a subfield of $N$.  
\end{proof} 

For simple f-near-rings  we get similar results with similar methods.

\begin{thm}\label{K1} 
 Let $(N,+,*)$ be a simple f-near-ring with DCCN and $U$ its set of units. If $N$ is a ring with $|U|=1$, then $N$ is isomorphic to a direct product of the fields $\Bbb{Z}_2$. If $N$ is a ring with $|U| \geq 2$, then $N$ is a field. In case $N$ is not a ring $N$ is a near-field.
\end{thm}
\begin{proof} From Theorem \ref{Lem1} and Proposition \ref{Ring1} the results concerning the ring case follow. So, we suppose that $N$ is not a ring in the following. 

As in the proof of Theorem \ref{Theorem1} (by \cite[Theorem 4.30]{Pilz1}) there is an element $k \in N$ such that $km_1=m_2$, $km_2=m_1$ and $km_i=m_i$ for $i \geq 3$, in case such an $i \geq 3$ exists, where the elements $m_i$, $i \in \{1, \ldots ,l\}$ ($l$ a natural number), are a set of representatives of the non-zero equivalence classes of $\sim$ in $M$. As in the proof of Theorem \ref{Theorem1} one gets that $k$ is a unit in $N$. Also, by \cite[Theorem 4.30]{Pilz1} there is a non-zero element $n \in N$ such that $nm_1=m_1$, $nm_2=m_1$, $nm_i=m_i$ for $i \geq 3$, in case such an $i \geq 3$ exists.  We are now assuming that $l \geq 2$.

By multiplying the set of equations $km_1=m_2$, $km_2=m_1$ and, if $l \geq 3$, $km_i=m_i$ for $i \geq 3$ with $n$  we get the following:
\begin{enumerate} \item $nkm_1=nm_2=m_1=nm_1$ and so, $nk-n \in (0:m_1)$. 
\item $nkm_2=nm_1=m_1=nm_2$ and so, $nk-n \in (0:m_2)$.
\item for $i \geq 3$ we get $nkm_i=nm_i$ and so, $nk-n \in (0:m_i)$.
\end{enumerate}
 Thus, for each $i \in \{1, \ldots ,l\}$ we have $nk-n \in (0:m_i)$. By the fact that the elements $m_i$ are a full set of representatives of the non-zero equivalence classes w.r.t. $\sim$ we get $nk-n \in (0:M)=\{0\}$. Thus, $nk=n$ and since $n \neq 0$ is a non-zero fixed point and $k$ is a unit we must have $k=1$. Thus, $m_1=m_2$ and this contradicts the fact that $m_1 \neq m_2$. Hence we must have $l=1$, so there is only one non-zero equivalence class w.r.t. $\sim$ in $M$. As done in the proof of Theorem \ref{Theorem1} for the case of one equivalence class one shows that $N$ is a near-field. 
\end{proof}

The results of the Theorems \ref{Theorem1} and \ref{K1} can be extended to near-rings which are $J_2$-semisimple. For a near-ring $N$, $J_2(N)$ is the intersection of all the annihilators of the $N$-groups of type $2$ (see \cite[Definition 5.1]{Pilz1}) and is an ideal of the near-ring such that the factor near-ring $N/J_2(N)$ becomes a subdirect product of $2$-primitive near-rings. When $N$ has DCCN this subdirect product will be a direct product of simple near-rings by \cite[Theorem 5.31]{Pilz1}. Note that in case of a  zero symmetric near-ring $N$ with identity and DCCN we always have $N \neq J_2(N)$ (see \cite[Proposition 5.43]{Pilz1}). We now study  f- and a-near-rings $N$ with DCCN and $J_2(N)=\{0\}$ and exclude the ring case, for this situation is completely clear.

\begin{thm}\label{LA1}  Suppose that $N$ is an f-near-ring or a-near-ring with DCCN which is not a ring. 
If $J_2(N)=\{0\}$ then $N$ is a near-field or $N  \cong M_0(\Bbb{Z}_3)$.
\end{thm}
\begin{proof} Let $J_2(N)=\{0\}$. By \cite[Theorem 5.31]{Pilz1} $N=N/J_2(N)$ is a direct product of $2$-primitive near-rings $N_i$, $i \in I$, $I$ a suitable index set and $\mid I \mid = k$, $k$ a positive integer. By \cite[Theorem 3.43]{Pilz1}, there are orthogonal idempotents $e_1, \ldots , e_k$ such that $e_1+ \ldots + e_k=1$, each $e_i \in N_i$ being the identity element in $N_i$.  

First assume that the identity is the only unit in $N$. Each of the $N_i$, $i \in I$, is a $2$-primitive near-ring with DCCN and an identity element $e_i$. Thus, for $i \in I$, $N_i \cong M_{G^{0}}(M)$ where $G^{0}:=Aut_N(M)\cup \{\overline 0\}$, $\overline 0$ being the zero function on $M$ and $G:=Aut_N(M)$, by \cite[Theorem 4.52]{Pilz1}. Since we assume that $N$ has a single unit, it follows that $N_i$ has only the unit $e_i$. The DCCN implies that $G$ has finitely many orbits on $M$, by \cite[Corollary 4.59]{Pilz1}. Let $j_1, j_2, \ldots , j_n$ ($n$ a natural number) be a complete set of orbit representatives of the action of $G$ on $M$. Suppose that $n \geq 2$. Let $f: \{j_1, \ldots , j_n\} \longrightarrow M$ be a function such that $f(j_1)=j_2$, $f(j_2)=j_1$ and $f(j_i)=j_i$, else. Then, the extension $\overline{f}$ in $M_{G^{0}}(M)$ (for the construction of this extension see the proof of Theorem \ref{Theorem1} or \cite[Theorem 4.28]{Pilz1}) is again a unit in $N_i$ which is not the identity. Consequently, 
we must have that $G$ has a single orbit in $M$. Suppose now that $|G| \geq 2$ and let $id \neq g \in G$. Let $h(j_1)=g(j_1)$. Then the extension of $h$ in $M_{G^{0}}(M)$ is a non-identity unit. This contradicts the fact that $N_i$ only has the identity as single unit. From this we see that $G=\{id\}$ and $M_{G^{0}}(M)=\Bbb{Z}_2$. Thus, $N$ is a ring.

So, we have that $|U| \geq 2$. Hence, there is a non-identity unit in $N$ and so, there is a near-ring $N_j$, $j \in I$ with another unit element $u_j \neq e_j$. 
W.l.o.g.  $j=1$. Since $N$ is the direct product of the near-rings $N_i$, $i \in I$, it is easily seen that $u_1+e_2+\ldots +e_k$ is another unit in $N$. 

Now assume that $N$ is an f-near-ring. 
Since $N_iN_j=\{0\}$ for $i \neq j$, for each  $j \geq 2$,  we use left distributivity over direct sums \cite[Theorem 2.29]{Pilz1} to see that $e_j(u_1+e_2+\ldots +e_k)=e_j$, so $e_j$ is a fixed point. 
By assumption, we now must have $e_j=0$ for $j \geq 2$. Consequently, $N=N_1$ is a $2$-primitive near-ring. By \cite[Corollary 4.47]{Pilz1}, $N$ is simple. The result now follows from Theorem \ref{K1}.

Now assume that $N$ is an a-near-ring.
Let $u_1+e_2+\ldots +e_k$ be a non-identity unit in $N$. Then $(u_1+e_2+\ldots +e_k)-(e_1+e_2+\ldots +e_k)=u_1-e_1 \in N_1$ is another unit in $N$ by assumption. Since $N_1$ is an ideal in $N$, containing the unit $u_1-e_1$ it follows that $N=N_1$. So, $N$ is a $2$-primitive and consequently a simple near-ring, so the result follows from Theorem \ref{Theorem1}.
\end{proof}

The proof of Theorem \ref{LA1} contains a result concerning near-rings with identity with DCCN which have the identity element as their single unit.
\begin{prop}\label{NRU1} Let $N$ be a zero symmetric near-ring with identity and DCCN which contains only the identity element as its single unit. If $J_2(N)=\{0\}$, then $N$ is a ring isomorphic to a direct product of the fields $\Bbb{Z}_2$.
\end{prop}
\begin{proof} This result follows from the proof of Theorem \ref{LA1} where we considered the case that a semisimple near-ring $N$ with identity has only a single unit, where the further assumption that $N$ is an a- or f-near-ring is not needed.
\end{proof}

\subsection{The order of the group of units}\label{Size}
In this section we consider a situation where f-near-rings are a-near-rings. 

Let $N$ be a finite f-near-ring and let $2 \leq k:=|U|$ be the size of the group of units. 
By Theorem \ref{CorA} we know that $|N| = p^n$ for some prime number $p$ and natural number $n$. 
Since $U$ acts as a group of fpf automorphisms we must have that $|U|$ divides $|N|-1$, so $k | p^n-1$. 

We now assume that $N$ is not a near-field. So, there must exist an ideal $I$ in $N$ by Theorem \ref{K1}. Let $m:=|I|$. 
Since $(I,+)$ is a subgroup of $(N,+)$ we must have $m=p^l$ for some natural number $l<n$. 
On the other hand, $I*U \subseteq I$ since $I$ is an ideal of the near-ring $N$. This implies that $U$ also acts as a group of fpf automorphisms on $(I,+)$. Consequently, we have that $k | p^l-1$ and so $k | gcd(p^n-1, p^l-1)=p^{gcd(n,l)}-1$ (see for example \cite[Hilfssatz B.1]{Wae}). 
If $n$ is a prime we now have that $k | p-1$. On the other hand, $U$ contains at least $p-1$ elements, namely the non-zero elements of the near-ring generated by $1$ (see the proof of Theorem \ref{CorA}). So we must have $|U|=p-1$.

Therefore we have shown the following.
\begin{thm}\label{PK1} Let $N$ be a finite f-near-ring. We assume that $N$ is not a near-field and $|N| = p^q$, $p, q$ being prime numbers. 
Then, $|U|=p-1$ and $(U \cup \{0\},+,*)$ is a field isomorphic to $\Bbb{Z}_p$. So, $N$ is an {a-near-ring}.
\end{thm}

Thus we can see that the near-ring constructed in Example \ref{UnitEx1}  is a minimal example of an f-near-ring that is not an a-near-ring.
The lowest integers $p^q$ with $q$ not prime are $2^{4}=16$, $2^{6}=64$ and $3^{4}=81$.
A search in SONATA showed that no f-near-ring that is not an a-near-ring of order $2^{4}$ exists.
So Example \ref{UnitEx1} is the smallest proper f-near-ring by ``dimension'' of the additive group.
A proper f-near-ring of order $2^{6}$ must have the group of units of order 3 or 7, as the units must also act fpf on a niontrivial ideal.
However cyclic groups of units of these orders do not occur without being additively closed. 
Thus Example \ref{UnitEx1} is also  the smallest proper f-near-ring by size.

In general, the structure of $U$ is not like in the theorem above. This follows easily from Theorem \ref{Tim1} which shows that we can construct f-near-rings of size $|F^k|$, $F$ a given near-field and $k$ a natural number where the units have size $|F \setminus \{0\}|$. In the next subsection we will construct all f-near-rings of size $p^{2}$ which are not near-fields, $p$ a prime number, where clearly we have $|U|=p-1$. We close this subsection with a class of examples of af-near-rings of size $p^k$, $p$ a prime number and $k$ a natural number where we have $|U|=p-1$. The examples are obtained by the construction method of Theorems \ref{A1} and \ref{A2}.

\begin{example}\label{UnitEx}
Let $(M,+):=(\Bbb{Z}_p^k,+)$ be the $k$-dimensional direct sum of the group $(\Bbb{Z}_p,+)$ written as column vectors and $m:=(0,\ldots ,0,1)^t$. 
Let 
\begin{align*}
 S:=\{\left(\begin{array}{cccc}
0&\cdot&0&a_1\\ \cdot & \cdot & \cdot & \cdot \\ \cdot & \cdot & \cdot & \cdot \\0&\cdot & 0 & a_k
\end{array}\right) \in &\Bbb{Z}_p^{k\times k} \mid \exists i \in \{1, \ldots k-1\}: a_i \neq 0\} \\&\cup \{diag(\lambda, \ldots, \lambda) \in \Bbb{Z}_p^{k\times k} \mid \lambda \in \Bbb{Z}_p\}
\end{align*}
$S$ is a semigroup of group endomorphisms of $(M,+)$ via matrix multiplication $\cdot $ using the field multiplication in $\Bbb{Z}_p$. For each element $(m_1,\ldots , m_k)^t \in M$, there is a unique matrix $s \in S$ such that $(m_1, \ldots , m_k)^t=s\cdot  m^t$.

For elements $(m_1, \ldots , m_k)^t \in M$ and $(n_1, \ldots , n_k)^t \in M$ we define the multiplication $* $ as 
\begin{align*}
 (n_1, \ldots , n_k)^t* (m_1, \ldots , m_k)^t=(n_1, \ldots , n_k)^t* (s\cdot m):=s\cdot (n_1, \ldots , n_k)^t
\end{align*}
By Theorem \ref{A1} $(M,+,*)$ is an f-near-ring, $(M,+,*)$ is not a near-field and the units are the elements of the form $(0, \ldots ,0,a_k)^t$ with non-zero $a_k$, so there are $p-1$ many units in the near-ring.
Also $Gm = \left<m\right> \leq (M,+)$ so we have an af-near-ring.
\end{example}
The example given in Example \ref{UnitEx} can also be derived from Theorem \ref{Tim1}. We will point this out, using the notation of Theorem \ref{Tim1}. Take the field $(\Bbb{Z}_p,+,*)$, form the $k$-dimensional sum of $(\Bbb{Z}_p,+)$ to get the group $(V,+)$. For $i \in \{1, \ldots , k\}$ let $\alpha_i$ be the identity map. Then, $(V,+,*_1)$ as constructed in Theorem \ref{Tim1} gives the same near-ring as constructed in Example \ref{UnitEx}.

\subsection{Near-rings of order $p^{2}$}
We study f-near-rings $N$  of size $p^{2}$ that are not nearfields.  
From Theorem \ref{PK1} we know that this implies they are a-near-rings. 
We will see that the construction of Theorem \ref{Tim1} gives all such near-rings.

\begin{thm}\label{GT1} 
Let $(N,+,*)$ be an f-near-ring of order $p^{2}$ which is not a near-field. Then $N$ is also an \emph{a-near-ring}. 
Moreover, there exists a group $(V,+)$ isomorphic to $(N,+)$ and a map $\alpha: U_0 \rightarrow U_0$ with $\alpha(0)=0$ and  $\alpha$ an automorphism of the multiplicative group $(U,*)$ such that $(N,+,*)$ is isomorphic to the near-ring $(V,+,*_1)$, where the near-ring operations are defined as in Theorem \ref{Tim1}. 
\end{thm}
\begin{proof}That $U_0$ is a subfield of the near-ring of order $p$ follows from Theorem \ref{PK1}. According to Theorem \ref{LA1}, $J_2(N) \neq \{0\}$ and since $1 \in N$, $J_2(N) \neq N$ by \cite[Proposition 5.43]{Pilz1}. Let $n \in N\setminus U$. Then, $Nn \neq N$ is an $N$-group of size $p$. Since $N$ has an identity element, for any non-zero element $mn \in Nn$ we have $\{0\} \neq Nmn \subseteq Nn$. But $Nn$ is a group of order $p$, so we must have $Nmn=Nn$ and consequently the $N$-group $Nn$ is of type $2$. Since $J_2(N)$ is the intersection of the annihilators of all $N$-groups of type $2$ of a near-ring $N$ (see \cite[Definition 5.1]{Pilz1}), $J_2(N)Nn=\{0\}$. Since $1 \in N$ we see that $J_2(N)n=\{0\}$. Also, we have that $N=J_2(N)+U_0$ and $J_2(N) \cap U_0=0$. Consequently, for each $n \in N$ there is a unique $j \in J_2(N)$ and a unique $u \in U_0$ such that $n=j+u$. Let $u \in U$. Since $U_0$ is a group of order $p$, $u$ is additively generated by the identity element, so $u=1+\ldots +1$, with 
$k$ 
summands, say. We will use the notation $u=k\cdot1$ for that and whenever we have a natural number $k$ and an element $n \in N$, then $k\cdot n$ will mean $n+\ldots +n$ with $k$ summands. The product of two natural numbers $m_1, m_2$ will be also denoted by $m_1 \cdot m_2$. 

Since $J_2$ is a group of order $p$, the groups $J_2$ and $U_0$ are isomorphic by a group isomorphism $\psi$. 
We let $\psi(J_2)=U_0$. Note that the near-ring multiplication $*$ when restricted to the units of the near-ring is a field multiplication because $(U_0,+,*)$ is a field.

Let $u \in U$ and $0 \neq \psi^{-1}(1) \in J_2$. Then, $0 \neq \psi^{-1}(1)*u$ because $u$ is a unit. Also, $J_2(N)$ is an ideal of the near-ring and therefore,  $\psi^{-1}(1)*u \in J_2(N)$. Since $J_2(N)$ is a group of order $p$, any element in $J_2(N)$ is additively generated by $\psi^{-1}(1)$. So, there is a smallest natural number $d_u \in \{1, \ldots, p-1\}$ such that $\psi^{-1}(1)*u = \psi^{-1}(1)+ \ldots + \psi^{-1}(1)=d_u \cdot \psi^{-1}(1)$. Now consider that map $\alpha: U_0 \rightarrow U_0, 0 \mapsto 0, 0 \neq u \mapsto d_u \cdot 1$. $\alpha$ is well defined. Note that for a unit $u$ we cannot have $\alpha(u)=0$. Let $u_1, u_2 \in U$ and suppose that $\alpha(u_1)=\alpha(u_2)$. Then, $\psi^{-1}(1)*u_1 = \psi^{-1}(1)*u_2$ and by fixedpointfreeness of the action of the units we get $u_1=u_2$. So, $\alpha$ is injective and consequently by finiteness bijective. We now show that $\alpha$ is a multiplicative homomorphism of $(U, *)$. To this end let $u_1, u_2 \in U$. Then, by associativity of the near-ring multiplication we have $(\psi^{-1}(1)*u_1)*u_2=\psi^{-1}(1)*(u_1*u_2)$. $(\psi^{-1}(1)*u_1)*u_2=(d_{u_1}\cdot \psi^{-1}(1))*u_2=(\psi^{-1}(1)+\ldots + \psi^{-1}(1))*u_2$  with $d_{u_1}$ summands. $\psi^{-1}(1)*u_2=d_{u_2}\cdot \psi^{-1}(1)$ (here we have $d_{u_2}$ many summands of $\psi^{-1}(1)$). Using right distributivity and the fact that $J_2(N)$ has order $p$ we get $(\psi^{-1}(1)*u_1)*u_2= (d_{u_1}\cdot d_{u_2} (mod p)) \cdot \psi^{-1}(1)$. On the other hand, $\psi^{-1}(1)*(u_1*u_2)=d_{u_1*u_2}\cdot \psi^{-1}(1)$ and from that we see that $d_{u_1*u_2}=d_{u_1}\cdot d_{u_2} (mod p)$. But this precisely means that $\alpha$ is a multiplicative homomorphism. So, $\alpha$ is a zero preserving automorphism of the multiplicative group $(U,*)$, as required in the construction of Theorem \ref{Tim1}.

We now will show that w.r.t.\ the field multiplication $(U,*)$, $N$ is isomorphic to a near-ring as constructed in Theorem \ref{Tim1}. To this end let $(V,+,*_1)$ be a near-ring as constructed in Theorem \ref{Tim1}, where $(V,+)=(U_0 \times U_0,+)$, $+$ the field addition in $(U_0,+,*)$  and multiplication $*_1$ defined as $(x_1,x_2)*_1(0,a_2):=(x_1*\alpha(a_2), x_2*a_2)$ and if $a_1 \neq 0$, $(x_1,x_2)*_1(a_1,a_2):=(x_2*a_1, x_2*a_2)$. 

We now claim that the map $\beta: N \rightarrow V, j_1+u_1 \mapsto (\psi(j_1),u_1)$ is a near-ring isomorphism. Since $\psi(J_2)=U_0$ and $(N,+)$ is the direct sum of the groups $(J_2, +)$ and $(U_0,+)$ it is clear that $\beta$ is an isomorphism of groups, so it remains to show that $\beta$ is a multiplicative map. 
We distinguish two cases. First we calculate $\beta((j_1+u_1)*(j_2+u_2))$ with $j_2 \neq 0$. 
Let $n_1=j_1+u_1$ an arbitrary element in $N$ and $n_2=j_2+u_2$ a non-unit in $N$, thus, $j_2 \neq 0$. Then, using right distributivity of the near-ring, $n_1n_2=j_1(j_2+u_2)+u_1(j_2+u_2)$. 
 $j_1(j_2+u_2)=0$ because $n_2$ is a non-unit. 
Again by right distributivity and using that $1$ is the identity of the near-ring and $u_1 = k\cdot 1$, $u_1(j_2+u_2)=(1+\ldots +1)(j_2+u_2)=(j_2+u_2)+ \ldots + (j_2+u_2)=k\cdot (j_2+u_2)$. Since $(N,+)$ is an abelian group (by Theorem \ref{CorA}, anyhow any group of order $p^{2}$ is abelian), this is $k\cdot j_2 + k \cdot u_2 = (1+ \ldots +1)j_2 + (1+ \ldots + 1)u_2=u_1j_2+u_1u_2$. 
Thus we have that the units distribute from the left hand side and $n_1n_2=(j_1+u_1)(j_2+u_2)=u_1n_2=u_1j_2+u_1u_2$ in case $n_2$ is a non-zero non-unit. 
Thus, $\beta((j_1+u_1)*(j_2+u_2))=\beta(u_1*j_2+u_1*u_2)$. 
Since $\beta$ is an additive homomorphism, this gives $\beta(u_1*j_2)+\beta(u_1*u_2)=(\psi(u_1*j_2),0)+(0,u_1*u_2)$. 
Let $u_1=k\cdot 1$. 
Then the last expression equals $(\psi(k\cdot j_2),0)+(0,u_1*u_2)$. 
Now $\psi$ is an additive homomorphism and $k\cdot j_2=j_2+ \ldots + j_2$ with $k$ many summands. 
So we have $\beta((j_1+u_1)*(j_2+u_2))=(k\cdot \psi(j_2),0)+(0,u_1*u_2)$. 

On the other hand, $\beta(j_1+u_1)*_1\beta(j_2+u_2)=(\psi(j_1), u_1)*_1(\psi(j_2),u_2)=(u_1*\psi(j_2), u_1*u_2)=((k\cdot 1)*\psi(j_2),u_1*u_2)=(k \cdot \psi(j_2), u_1*u_2)$. So, in case $j_2 \neq 0$ we see that $\beta$ is a multiplicative map.

Now consider the second case, namely $\beta((j_1+u_1)*u_2)=\beta(j_1*u_2+u_1*u_2)=(\psi(j_1*u_2),u_1*u_2)$. On the other hand, $\beta(j_1+u_1)*_1\beta(u_2)=(\psi(j_1),u_1)*_1(0,u_2)=(\psi(j_1)*\alpha(u_2), u_1*u_2)$. In order to show that these expressions are the same, we have to compare the elements $\psi(j_1*u_2)$ and $\psi(j_1)*\alpha(u_2)$ and show that they are equal. Since $J_2$ is a cyclic group of order $p$, $j_1$ is a sum of the elements $\psi^{-1}(1)$, $j_1=\lambda \cdot \psi^{-1}(1)$, say with $\lambda \in \{0, \ldots ,p-1\}$. So, $j_1*u_2= (\lambda \cdot \psi^{-1}(1))*u_2=d_{u_2} \cdot (\lambda \cdot  \psi^{-1}(1))=(\lambda \cdot d_{u_2}) \cdot \psi^{-1}(1)$. Since $\psi$ is a homomorphism we get that $\psi(j_1*u_2)=\psi((\lambda\cdot d_{u_2}) \cdot \psi^{-1}(1))=(\lambda \cdot d_{u_2})\cdot \psi(\psi^{-1}(1))=(\lambda \cdot d_{u_2}) \cdot 1$. By definition of $\alpha$, $\alpha(u_2)=d_{u_2}\cdot 1$. So, $\psi(j_1)*\alpha(u_2)=\psi(\lambda \cdot \psi^{-1}(1))*(d_{u_2}\cdot 1)=(\lambda \cdot \psi(\psi^{-1}(1))*(d_{u_2}\cdot 1)$, 
using that $\psi$ is an additive homomorphism. This gives $(\lambda\cdot  d_{u_2})\cdot 1$. So, indeed $\psi(j_1*u_2)= \psi(j_1)*\alpha(u_2)$ and this proves that $\beta$ is an isomorphism.
\end{proof}

Hence we see that the construction method of Theorem \ref{Tim1} gives us all f-near-rings of order $p^{2}$, $p$ a prime number. 

\section{Conclusion}

In this paper we have looked at the situation for near-rings with their units obeying one or the other of two
special properties: acting fixed point freely (f-near-rings) and being additively closed (including the zero of the near-ring) (a-near-rings).
For rings with DCCL the situation is clear: all such rings are fields. 
It remains open whether infinite a-rings or f-rings exist that are not fields.
For proper near-rings, we have seen that more complex and interesting situations can arise.

We have investigated a-near-rings only in the case that the group of units is nontrivial.
We have seen that a-rings with trivial group of units are a direct product of copies of $\Bbb{Z}_2$.
It remains open what a-near-rings exist with trivial unit group and what can be said about them more than that their additive groups are elementary abelian 2-groups.

We have seen that  the additive group of a-near-rings are of prime exponent, while {f-near-ring} additive groups are elementary abelian.
It is shown that in many cases, for instance when the order of the nearring is $p^q$ for primes $p,q$, the units form a prime field.
Examples have been constructed  to show that all near-fields can arise as the units of an {af-near-ring}. 
On the other hand, it is unclear whether all fixedpointfree automorphism groups can appear as the units in an f-near-ring.

It has been shown that {f-near-rings} which are not near-fields are nonsimple and not $J_2$-semisimple. 
It remains open what special structures the 
$J_2$ radical possesses. We have seen that the radical has an important role in determining the structure of f-nearrings of order $p^{2}$. It is an open question to the authors if the Jacobson radical of type $2$ in the type of near-rings we study is always nilpotent and if the construction method of Theorem \ref{Tim1} also gives all {af-near-rings} of some higher prime power order. 
The question of deciding when f-near-rings constructed by Theorems \ref{A1}, \ref{Tim1} are isomorphic would be an interesting problem to consider.

\section{Acknowledgements}

We would like to thank Prof. G\"unter Pilz and our colleagues Wen-Fong Ke and Peter Mayr for input during the process of preparing this paper.

\end{document}